\documentclass[10pt,epsfig]{article}
\usepackage{amsmath}
\numberwithin{equation}{section}
\usepackage{amsthm}
 \usepackage{amssymb}
\usepackage{epsfig}
 \usepackage{leftidx}
\usepackage{graphicx}
\usepackage{hyperref}
\usepackage{amsfonts}
\usepackage{graphicx, color, epstopdf}
 \usepackage{cite}
 \usepackage{bm}
  \usepackage{xcolor}

\newtheorem{theorem}{Theorem}[section]
\newtheorem{lemma}{Lemma}[section]
\newtheorem{proposition}{Proposition}[section]
\newtheorem{corollary}{Corollary}[section]

\renewenvironment{proof}[1][Proof]{\noindent\textit{#1. } }{\hfill$\square$}
\newcommand{\Ass}[1]{\textbf{\upshape A#1}}
\newcommand{\R}{\mathbb{R}}

\newcommand{\taumax}{\tau}

\newcommand{\defeq}{:=}

\newcommand{\zd}{\,\mathrm{d}}

\newcommand{\brab}[1]{\big(#1\big)}

\def\d{\delta}

\def\R{\mathbb{R}}

\title{
Unconditionally optimal error estimate of a linearized variable-time-step  BDF2 scheme for  nonlinear parabolic equations
}

\date{September 1, 2021}
\author{Chengchao Zhao
	\thanks{Beijing Computational Science Research Center, Beijing, 100193, P.R. China (cheng\_chaozhao@csrc.ac.cn).}
	\and Nan Liu
	\thanks{School of Mathematics and Statistics, Wuhan University, Wuhan 430072, China (liunan@whu.edu.cn)}
\and Yuheng Ma
 \thanks {School of Mathematics and Statistics, Wuhan University, Wuhan 430072, China (yuhengma@whu.edu.cn)}
	\and Jiwei Zhang
	\thanks{School of Mathematics and Statistics, and Hubei Key Laboratory of Computational Science, Wuhan University, Wuhan 430072, China (jiweizhang@whu.edu.cn). He is supported partially by NSFC under grant 11771035.}
	}

\begin{document}
\maketitle
\begin{abstract}
In this paper we consider a linearized variable-time-step two-step backward differentiation formula (BDF2) scheme for solving nonlinear parabolic equations. The scheme is constructed by using the variable time-step BDF2  for the linear term and a Newton linearized method for the nonlinear term in time combining with a Galerkin finite element method (FEM) in space. We prove the unconditionally optimal error estimate of the proposed scheme under mild restrictions on the ratio of adjacent time-steps, i.e. $0<r_k < r_{\max} \approx 4.8645$ and on the maximum time step.  The proof involves  the discrete orthogonal convolution (DOC) and discrete complementary convolution (DCC) kernels, and the error splitting approach.   In addition, our analysis also shows that  the first level solution $u^1$ obtained by BDF1 (i.e. backward Euler scheme) does not cause the loss of global accuracy of second order. 
Numerical examples are provided to demonstrate our theoretical results.
 \par\textbf{keywords:}  two-step backward differentiation formula; the discrete orthogonal convolution  kernels; the discrete complementary convolution  kernels; error estimates; variable time-step; Unconditional error estimate; error splitting approach
\end{abstract}
\section{Introduction}
In this paper, we focus on the unconditionally optimal error estimate of a linearized second-order two-step backward differentiation formula (BDF2) scheme with variable time steps for solving the following general nonlinear parabolic equation \cite{F37,ZF38}:
 \begin{equation}\label{001}
\begin{aligned}
&\partial_t u=\Delta u +f(u),&&\bm{x}\in\Omega, t\in(0,T],\\
&u(\bm{x},0)=u_0(\bm{x}),&& \bm{x}\in \Omega,\\
&u(\bm{x},t)=0,&&\bm{x}\in\partial\Omega, t\in[0,T],
\end{aligned}
\end{equation}
where $\Omega\subset \mathbb{R}^d(d=1,2,3)$ is a bounded convex  domain. 
To construct variable time-step schemes, we first set the variable time levels $0=t_0<t_1<\cdots<t_N=T$,  the $k$th time-step size $\tau_k\defeq t_k-t_{k-1}$, the maximum step size
$\taumax\defeq\max_{1\le k\le N}\tau_k$, and the adjacent time-step ratios
$r_k = \tau_k/\tau_{k-1}, \; 2\leq k \leq N.$ Set $u^k=u(t_k), \nabla_\tau u^k := u^k - u^{k-1}, \; r_1\equiv 0$, and denote $U^n$ by the approximation of the exact solution $u(t_n)$.
The first step value $U^1$ is calculated by one-step backward difference formula  (BDF1), and the {other} $U^n\; (n>1)$ is calculated by BDF2 formula with variable time steps, which are respectively given as
\begin{equation} \label{2104241}
\mathcal{D}_1 U^1 = \frac1{\tau_1}\nabla_\tau U^1,\quad \mathcal{D}_2 U^n = \frac{1+2r_n}{\tau_n(1+r_n)} \nabla_\tau U^n - \frac{r_n^2}{\tau_n(1+r_n)} \nabla_\tau U^{n-1}.
\end{equation}
The nonlinear term is approximated by a linearized method, i.e.,  the Newton linearized method given as
\begin{align}\label{009}
f({u^n})\approx f(U^{n-1})+f'(U^{n-1})\nabla_\tau U^n.
\end{align}
Thus the semi-discrete BDF2 scheme with variable time steps to problem \eqref{001} is given as
\begin{align} \label{018}
&\mathcal{D}_2 U^n = \Delta U^n  + f(U^{n-1})+f'(U^{n-1})\nabla_\tau U^n,&& \text{for} \quad 1\leq n\leq N,\\
&U^0(\bm{x})=u_0(\bm{x}),&& \bm{x} \in \Omega,\\
&U^n(\bm{x})=0,&& \bm{x} \in\partial\Omega,\quad n=0,1,\cdots,N.
\end{align}
Here the BDF1 and BDF2 formulas have been written as a unified convolution form of
\begin{equation} \label{002}
\mathcal{D}_2 U^n = \sum_{k=1}^n b_{n-k}^{(n)} \nabla_\tau U^k, \quad \quad n\geq 1
\end{equation}
 by taking $b_0^{(n)} =  (1+2r_n)/(\tau_n(1+r_n)), \; b_1^{(n)}=  - r_n^2/(\tau_n(1+r_n))\; \text{and} \; b_j^{(n)} = 0 \; \text{ for } 2\leq j \leq n-1.$

{For the spatial discretization,} let $\mathcal{I}_h$ be the quasiuniform partition of $\Omega$ with triangles in $\R^2$ or tetrahedra  $T_i \; (i=1,\cdots,M)$ in $\R^3$, see \cite{lwz17,T06}. Denote  the spatial mesh size by $h=\max_{1\leq i\leq M}\{\text{diam} \ T_i\}$ and the finite-dimensional subspace of $H^1_0(\Omega)$ by $V_h$ which consists of piecewise polynomials of degree $r \;(r\geq 1)$ on $\mathcal{I}_h$.
Then, the Newton linearized Galerkin FEM BDF2 scheme with variable time steps is to find $U^n_h\in V_h$ such that
 \begin{align}\label{008}
(\mathcal{D}_2 U_h^n,v_h)=-(\nabla U_h^n, \nabla v_h)+(f(U_h^{n-1})+f'(U_h^{n-1})\nabla_\tau U_h^n,v_h), \quad \forall v_h\in V_h.
\end{align}

Efficiency, accuracy and  reliability are of key considerations in numerical analysis and scientific computing.  For the time-dependent models appeared in science and engineering,  a heuristic and promising method to improve efficiency without sacrificing accuracy is the time adaptive method. For instance, one may employ the coarse-grained or refined time steps based on the solutions changing slowly or rapidly  to capture the dynamics of the solutions. Another alternative approach is to employ high-order
methods in time to have the same accuracy with a relatively large time-step. In this paper, we consider the second order BDF2 scheme with variable time steps.

Due to its nice property ($A$-stable),  the BDF2 method with variable time steps has  been widely used in various models  to  obtain computationally efficient, accurate results \cite{CWYZ19,zz21t,WCF19,LTZ20,LSTZ21,LJZ20}. Much works have been carried out on its stability and error estimates \cite{zz21,lz20,G83,E05,CWYZ19,B98}, but the  analysis even for linear parabolic problems (i.e., problem \eqref{001} with $f(u) =0$) is already highly nontrivial and challenging as documented in the classic book \cite[Chapter 10]{T06}. One also refers to the details in \cite{G91,G95,G83,L82,P97}.  In particular, with the energy method and under a ratio condition $0< r_k \leq  1.868$, Becker \cite{B98} presents that the variable time-step BDF2 scheme for a linear parabolic problems is  zero-stable. After that, Emmrich \cite{E05} gives the similar result with $0< r_k \leq 1.91$.  Recently, a promising work in \cite{CWYZ19} introduce a novel generalized discrete Gr\"onwall-type inequality for the Cahn-Hilliard equation to obtain
an energy stability with $0< r_k \leq 3.561$.  The works in \cite{lz20,zz21} consider linear parabolic equations based on DOC kernels under $0\leq r_k \leq 3.561$ \cite{lz20} and $0<r_k\leq 4.8645$ \cite{zz21}, respectively. In addition, the  robust second-order convergence is further analyzed in \cite{zz21}. The robustness here means the optimal second-order accuracy holds valid only requiring the time-step sizes  $0<r_k\leq 4.8645$.

For the variable-time-step BDF2 method applied to nonlinear parabolic equations, there has a great progress on the error estimates for the Cahn-Hilliard equation \cite{CWYZ19}, molecular beam epitaxial model without slope selection \cite{zz21t,LSTZ21},  phase field crystal model \cite{LJZ20}, and references therein. Among  all the  numerical methods in \cite{CWYZ19,zz21t,LSTZ21,LJZ20},  the implicit schemes are utlized to deal with the nonlinear terms. Although the implicit schemes are unconditionally stable,  they need to solve a nonlinear algebraic system in each time level.  To circumvent the extra computational cost for iteratively solving the nonlinear algebraic system, a popular approach is to employ the linearized methods to approximate the nonlinear terms. Combining the variable-time-step BDF2 for linear terms with a linearized method for nonlinear terms, it is natural to ask if the proposed implicit-explicit schemes still remain unconditionally stable. In addition, the focus of this paper is on the general nonlinearity $f$, which also brings the extra analysis difficulty comparing with the phase-fields models studied in \cite{CWYZ19,zz21t,LSTZ21,LJZ20},  since the good property of the energy dissipation for the phase-fields models does not hold any more for general nonlinearity.

In this paper, we aim to address the unconditionally optimal error estimate  $\mathcal{O}(\tau^2+h^{r+1})$ of the linearized scheme \eqref{008} in the sense of $L^2$-norm under the following two conditions:
  \begin{itemize}
\item [\Ass{1} :]  $0< r_k \leq r_{\max}-\delta$ for any small constant $0<\delta < r_{\max}\approx 4.8645$  and  $2\leq k \leq N$,
where $r_{\max} = \frac{1}{6} \left(\sqrt[3]{1196\!-\!12\sqrt{177}}+\sqrt[3]{1196\!+\!12\sqrt{177}}\right)+\frac43$ is the root of equation $x^3 = (1+2x)^2$;
\item [\Ass{2} :] there exists a constant $\hat{C}$ independent of $\tau$ and $N$ such that the maximum time step size $\tau$ satisfies $\tau\leq \hat{C}\frac{1}{\sqrt{N}}$.
\end{itemize}

The proof of the error estimate is established by the introduction of the temporal-spatial error splitting approach and the concepts of DOC and DCC kernels.  The error splitting approach, developed in \cite{LS12}, is used to overcome the unnecessary  restrictions of the temporal and spacial mesh sizes. The main idea of the error splitting approach is to first consider the boundedness of $\|U^n\|_{H^2}$ for the solution to semi-discrete equation \eqref{018}, and then to obtain the estimate of $\|U_h^n\|_{L^\infty}$ for the fully discrete equation \eqref{008} by using the following strategy
\begin{align}
\|U_h^n\|_{L^\infty}&\leq\|R_hU^n\|_{L^\infty}+\|R_hU^n-U_h^n\|_{L^\infty}\nonumber\\
&\leq C\|U^n\|_{H^2}+Ch^{-d/2}\|R_hU^n-U_h^n\|_{L^2}\label{EQ_1444}\\
&\leq C+Ch^{-d/2}h^2,\nonumber
\end{align}
where $R_h$ is the projection operator.

The concepts of DOC and DCC kernels, developed in \cite{lz20,zz21}, are used to present the stability and convergence analysis under a mild restriction on the ratio of adjacent time-steps, i.e., \Ass{1}.  The condition \Ass{1} is the mildest restriction in  the current literatures, as far as we know, for the variable time-step BDF2 method. On the other hand, due to the usage of a linearized scheme to approximate the general nonlinearity, our optimal error estimate in time  suffers from another restriction on the maximum time step, i.e., \Ass{2}.  Noting that the condition \Ass{2}  only involves the maximum time step, hence it is mild and acceptable/reasonable for the practical simulations to have the convergence order.

For general nonlinearity, our improvements are twofold:  (i) comparing with the fully implicit schemes developed \cite{CWYZ19,zz21t,LSTZ21,LJZ20},   a linearized scheme is considered in this paper, and its first rigorous proof of unconditionally optimal convergence is presented under mild restrictions on \Ass{1} and \Ass{2};  (ii) the linearized BDF2 scheme \eqref{008} still has the second-order accuracy in time as the first order BDF1 method is used only once to compute the first step value.

The remainder of this paper is organized as follows. In section \ref{Section_2}, we present several important   properties of the DOC kernels and DCC kernels which play a key role in our stability and convergence analysis. The unconditional $L^\infty$ boundedness of numerical solutions of the fully discrete scheme is derived in section \ref{Section_3} and then the unconditionally optimal $L^2$-norm error estimate is obtained in section \ref{Section_4}. In section \ref{Section_5}, numerical examples are provided to confirm our theoretical
analysis.

\section{Setting }\label{Section_2}
In this paper,  we assume there exists a constant $\mathcal{M}$ such that the solution of problem \eqref{001}  satisfies
\begin{align}
\|u_0\|_{H^{r+1}}&+\|u\|_{L^\infty((0,T);H^{r+1})}+\|\partial_t u\|_{L^\infty((0,T);H^{r+1})}\nonumber \\
&+{\|\partial_{tt} u\|_{L^\infty((0,T);L^2)}+\|\partial_{ttt} u\|_{L^\infty((0,T);L^2)}\leq \mathcal{M}},\   r\geq 1.\label{EQ_uregularity }
\end{align}
 \subsection{The properties of DOC and DCC kernels}
We here present the definitions and properties of DCC and DOC kernels, which play crucial roles in our analysis to overcome the difficulties resulted from the variable time steps. The DCC kernels introduced in \cite{llwsz18,llz18}  are defined by
\begin{equation}
\sum_{j=k}^n p_{n-j}^{(n)} b_{j-k}^{(j)} \equiv1, \quad  \forall  1\leq k \leq n,  \; 1 \leq n \leq N,\label{Def1}
\end{equation}
which satisfy
 \begin{equation} \label{did1}
\sum_{j=1}^np_{n-j}^{(n)}\mathcal{D}_2 u^j = \sum_{j=1}^np_{n-j}^{(n)} \sum_{l=1}^j b_{j-l}^{(j)} \nabla_\tau u^l = \sum_{l=1}^n\nabla_\tau u^l \sum_{j=l}^np_{n-j}^{(n)}b_{j-l}^{(j)}  =  u^n - u^0, \quad  \forall n\geq 1.
\end{equation}
The DOC kernels in  \cite{lz20}  is defined by
\begin{align}\label{005}
& \sum_{j=k}^n \theta_{n-j}^{(n)} b_{j-k}^{(j)} = \delta_{nk},  \quad  \forall  1\leq k \leq n,  \; 1 \leq n \leq N,
\end{align}
where the Kronecker delta symbol $ \delta_{nk}$ holds $ \delta_{nk} = 1$ if $n=k$ and $ \delta_{nk} = 0$ if $n\neq k$.
By exchanging the summation order, it is straightforward to verify that the DOC kernels satisfy
 \begin{equation} \label{006}
\sum_{j=1}^n\theta_{n-j}^{(n)}\mathcal{D}_2 u^j= \sum_{l=1}^n\nabla_\tau u^l \sum_{j=l}^n\theta_{n-j}^{(n)}b_{j-l}^{(j)}  =  u^n - u^{n-1}, \quad  1\leq n \leq N.
\end{equation}
Set $p^{(n)}_{-1}:=0,\forall n\geq 0$. The DCC and DOC kernels have the following relations (see \cite[Proposition 2.1]{zz21})
\begin{align}
p^{(n)}_{n-j}=\sum_{l=j}^n \theta^{(l)}_{l-j},\quad
\theta^{(n)}_{n-j} = p^{(n)}_{n-j}-p^{(n-1)}_{n-1-j},\quad \forall 1\leq j\leq n.\label{EQ_ptheta}
\end{align}

We now present several useful lemmas with details in \cite{lz20,zz21,zz21t}.
\begin{lemma}[\!\!\cite{zz21t}]\label{lemma01}
Assume the time step ratio $r_k$ satisfies \Ass{1}.    For any real sequence $\{w_k\}_{k=1}^n$ and any given small constant $0<\delta < r_{\max} \approx 4.8645$,  it holds
\begin{align}
&2w_k \sum_{j= 1}^k b^{(k)}_{k-j} w_j \geq \frac{r_{k+1}\sqrt{r_{\max}}}{(1+r_{k+1})}\frac{w_k^2}{\tau_k} - \frac{r_k\sqrt{r_{\max}}}{ (1+r_k)}\frac{w_{k-1}^2}{\tau_{k-1}}+\frac{\delta  w_k^2}{20\tau_k}, \quad k \geq 2,\label{EQ_bpositiveup} \\
&2\sum_{k=1}^nw_k \sum_{j= 1}^k b^{(k)}_{k-j} w_j \geq \frac{\delta }{20}\sum_{k=1}^n\frac{w_k^2}{\tau_k}\geq 0, \quad \text{for } n \geq 1.\label{EQ_bpositive}
\end{align}
\end{lemma}

\begin{corollary}[\!\!\cite{zz21}]\label{corollary01}
If the condition in lemma \ref{lemma01} are satisfied, then we have the  following inequality
\begin{align}
&\sum_{k=1}^nw_k \sum_{j= 1}^k \theta^{(k)}_{k-j} w_j \geq  0, \quad \text{for } n \geq 1.\label{EQ_bpositive1}
\end{align}
\end{corollary}
\begin{lemma}[\!\!\cite{lz20}]\label{lemma03}
  The DOC kernels $\theta_{n-j}^{(n)}$ {satisfy} 
  \begin{align} \label{011}
 \theta_{n-j}^{(n)}&>0 \; \text{ for }\; 1\leq j\leq n, \qquad\text{ and } \qquad \sum_{j=1}^{n} \theta_{n-j}^{(n)} = \tau_n \; \text{ for } \; n\geq 1.
\end{align}
\end{lemma}
\begin{proposition}\emph{(\!\!\!\cite[Proposition 2.2]{zz21}).} \label{Proposition_Pestimate}
{The DCC kernels $p^{(n)}_{n-k}$ defined in \eqref{Def1} satisfy}
\begin{align}
&\sum_{j=1}^n p^{(n)}_{n-j}=\sum_{k=1}^n\sum_{j=1}^k \theta^{(k)}_{k-j} = t_n, \qquad
p^{(n)}_{n-j}\leq2\tau. \label{EQ_pnkr}
\end{align}
\end{proposition}

\section{The $L^\infty$ boundedness of the fully discrete solution $U^n_h$}\label{Section_3}
The key point in this paper, to deal with the nonlinear term in \eqref{008},  is the $L^\infty$ boundedness of the numerical solution $U^n_h$. To this end, we present several useful lemmas as follows, including the discrete Gr\"{o}nwall inequality and the temporal  consistency errors. For convenience, here and after we denote $\|\cdot\|:=\|\cdot\|_{L^2}$.
\begin{lemma}[Discrete Gr\"{o}nwall inequality] \label{lemma04}
Assume $\lambda>0$ and the sequences $\{v_j\}_{j=1}^N$ and $\{\eta_j\}_{j=0}^N$ are nonnegative. If
$$
v_n \leq \lambda \sum_{j=1}^{n-1} \tau_j v_j + \sum_{j=0}^n \eta_j, \quad \text{for} \quad 1\leq n \leq N,
$$
then it holds
$$
v_n \leq \exp\big(\lambda t_{n-1}\big) \sum_{j=0}^n \eta^j, \quad \text{for} \quad 1\leq n \leq N.
$$
\end{lemma}
Lemma \ref{lemma04} can be proved by the  standard induction hypothesis and we omit it here.
\begin{lemma}[\!\!\cite{lwz19}]\label{lemma05}
Assume the regularity condition \eqref{EQ_uregularity } holds  and the nonlinear function $f=f(u)\in C^2(\mathbb{R})$. Denote the local truncation error by
\begin{align}\label{012}
R_f^j=f(u^j)-f(u^{j-1})-f'(u^{j-1})\nabla_\tau u^j, \quad 1\leq j\leq N.
\end{align}
Then we have the following estimates of the  truncation error
\begin{eqnarray} \label{013}
&&  \|R_f^j\|\leq C_f \tau_j^2, \quad 1\leq j\leq N,
\end{eqnarray}
where $C_f:=\frac12 C_\Omega \sup_{|u|\leq C_\Omega\mathcal{M}}|f''(u)|\mathcal{M}^2$.
\end{lemma}

\begin{lemma}\label{lemma202147}
Assume the regularity condition \eqref{EQ_uregularity } holds, the truncation error $R_t^j$ has the properties:
\begin{align}
&\|R_t^1\| \leq \frac{\mathcal{M}}2\tau_1, \quad
\|R_t^j\| \leq \frac32\mathcal{M}\tau_j\tau,\quad 2 \leq j\leq N.
\end{align}
\end{lemma}

The proofs of Lemmas \ref{lemma05} and \ref{lemma202147} mainly use the Taylor expansion,  and are left to Appendix  \ref{App61} and \ref{App62} for brevity.

\subsection{Analysis of a semi-discrete scheme}
{As indicated in \eqref{EQ_1444}},   we first consider  the  boundedness of the solution to semi-discrete scheme \eqref{018}  in this subsection,  and leave the boundedness of the error $\|U^n-U_h^n\|_{L^\infty}$ in next subsection.

 Let $e^n=u^n-U^n \;(n=0,1,\ldots,N)$. Subtracting (\ref{018}) from (\ref{001}), one has
 \begin{align}\label{020}
\mathcal{D}_2 e^n=\Delta e^n  +R_t^n+R_f^n+E_1^n,
\end{align}
where
 $
E_1^n=f(u^{n-1})+f'(u^{n-1})\nabla_\tau u^n- f(U^{n-1}) - f'(U^{n-1})\nabla_\tau U^n.
$
\begin{theorem}\label{theorem01}
Assume the conditions \Ass{1} and \Ass{2} and the regularity condition \eqref{EQ_uregularity } hold,  and the nonlinear function $f\in C^2(\mathbb{R})$.   Then the semi-discrete system (\ref{018}) has a unique solution $U^n$. Moreover,  there exists an $\tau^{**}>0$  such that  the following estimates hold for all $\tau\leq\tau^{**}$
\begin{align}
&\|e^n\|_{H^1} \leq  C_1\tau^\frac32,\label{022}\\
&\|U^n\|_{L^\infty} \leq C_\Omega  \|U^k\|_{H^2}\leq  C_\Omega  (\mathcal{M}+1),\label{024}\\
&\sum_{j=1}^n\|\nabla_\tau e^j\|_{H^2} \leq C_{2}. \label{023}
\end{align}
More precisely, in (\ref{022}), we have $\|e^n\| \leq C_3\tau^2$ and $\|\nabla e^n\|\leq C_4\tau^\frac32$, where  
$C_i\; (i= 1,2,3,4)$ are positive constants independent of $\tau$.
\end{theorem}
\begin{proof}
Set
\begin{align}\tau^{**}=\min\{1, 1/(8C_5),1/C_8\},\label{EQ_tau**}\end{align}
where $C_5,C_8$ will be determined later. Noting that, at each time level, (\ref{018}) is a linear elliptic problem, it is easy to obtain the existence and uniqueness of solution $U^n$. We here use the mathematical induction to prove  (\ref{022}) and (\ref{024}), which obviously hold for the initial level $n=0$. Assume   (\ref{022}) and (\ref{024}) hold for $0\leq n\leq k-1\;(k\leq N)$. Based on the boundedness of $\|U^{n-1}\|_{L^\infty}$ and $\|u^{n-1}\|_{L^\infty}$ for $1\leq n\leq k$, we have
\begin{align}
\|E_1^n\|&= \|f(u^{n-1})+f'(u^{n-1})\nabla_\tau u^n-(f(U^{n-1})+f'(U^{n-1})\nabla_\tau U^n)\|\nonumber\\
&\leq \|f(u^{n-1})-f(U^{n-1})\|+\|(f'(u^{n-1})-f'(U^{n-1}))u^n\| +\|f'(U^{n-1})(u^n-U^n)\|\nonumber\\
&\quad+\|(f'(u^{n-1})-f'(U^{n-1}))u^{n-1}\| +\|f'(U^{n-1})(u^{n-1}-U^{n-1})\|\nonumber\\
&\leq C_5(\|e^{n-1}\|+\|e^{n}\|), 
\label{027}
\end{align}
where
$C_5=2\sup_{\lvert v\rvert\leq C_\Omega (\mathcal{M}+1)}\lvert f'(v)\rvert + 2C_\Omega \mathcal{M}\sup_{\lvert v \rvert\leq C_\Omega (\mathcal{M}+1)}\lvert f''(v)\rvert .$

We now prove that (\ref{022}) and (\ref{024}) hold at $n=k$. Set  $n=j$ in \eqref{020}.  Multiplying  $\theta_{l-j}^{(l)}$  to both sides of \eqref{020}, and summing the resulting from 1 to $l$, we have
 \begin{align}\label{029}
\nabla_\tau e^l=\sum_{j=1}^l\theta_{l-j}^{(l)}(\Delta e^j +E_1^j) + \sum_{j=1}^l\theta_{l-j}^{(l)}(R_t^j+R_f^j),
\end{align}
where the property of DOC kernels \eqref{006} is used. Then taking the inner product with $e^l$ on both sides of  (\ref{029}), and summing the resulting equality from 1 to $k$, one has
\begin{equation}\label{030}
\begin{split}
\sum_{l=1}^{k}(\nabla_\tau e^l,e^l)&=\sum_{l=1}^{k}\sum_{j=1}^l\theta_{l-j}^{(l)}(\Delta e^j +E_1^j,e^l) + \sum_{l=1}^{k}\sum_{j=1}^l\theta_{l-j}^{(l)}(R_t^j+R_f^j,e^l).\\
\end{split}
\end{equation}
{Applying \eqref{EQ_bpositive},  integration by parts and  the inequality $2(a-b)a\geq a^2-b^2$ to \eqref{030}}, we have
\begin{align*}
\|e^k\|^2-\|e^{0}\|^2&\leq2\sum_{l=1}^{k}\sum_{j=1}^l\theta_{l-j}^{(l)}(E_1^j,e^l)+2\sum_{l=1}^{k}\sum_{j=1}^l\theta_{l-j}^{(l)}(R_t^j+R_f^j,e^l)\\
&\leq 2\sum_{l=1}^{k}\|e^l\|\sum_{j=1}^l\theta_{l-j}^{(l)}\|E_1^j\|+2\sum_{l=1}^{k}\|e^l\|\|\sum_{j=1}^l\theta_{l-j}^{(l)}(R_t^j+R_f^j)\|,
\end{align*}
which together with  the inequality \eqref{027} and $\|e^0\|=0$ produces
\begin{equation}\label{031}
\|e^k\|^2
\leq 2C_5\sum_{l=1}^{k}\|e^l\|\sum_{j=1}^l\theta_{l-j}^{(l)}(\|e^j\|+\|e^{j-1}\|)+2\sum_{l=1}^{k}\|e^l\|\|\sum_{j=1}^l\theta_{l-j}^{(l)}(R_t^j+R_f^j)\|.
\end{equation}
Choosing an integer $k^*(0\leq k^*\leq k)$ such that $\|e^{k^*}\|=\max_{0\leq i\leq k}\|e^{i}\|$, then the inequality \eqref{031} yields
\begin{align}
\|e^{k}\|\|e^{k^*}\|\leq\|e^{k^*}\|^2&\leq 4C_5\|e^{k^*}\|\sum_{l=1}^{k^*}\tau_l\|e^l\|+2\|e^{k^*}\|\sum_{l=1}^{k^*}\|\sum_{j=1}^l\theta_{l-j}^{(l)}(R_t^j+R_f^j)\| \nonumber\\
&\leq 4C_5\|e^{k^*}\|\sum_{l=1}^{k}\tau_l\|e^l\|+2\|e^{k^*}\|\sum_{l=1}^{k}\|\sum_{j=1}^l\theta_{l-j}^{(l)}(R_t^j+R_f^j)\|. \label{032}
\end{align}
{where \eqref{011} is used.}
Thus, we arrive at
\begin{equation}\label{033}
\begin{split}
\|e^{k}\|&\leq
4C_5\sum_{l=1}^{k}\tau_l\|e^l\|+2\sum_{l=1}^{k}\|\sum_{j=1}^l\theta_{l-j}^{(l)}(R_t^j+R_f^j)\|.
\end{split}
\end{equation}
It follows from  \eqref{EQ_tau**}  that
\begin{equation}\label{034}
\begin{split}
\|e^{k}\|&\leq
8C_5\sum_{l=1}^{k-1}\tau_l\|e^l\|+4\sum_{l=1}^{k}\|\sum_{j=1}^l\theta_{l-j}^{(l)}(R_t^j+R_f^j)\|.
\end{split}
\end{equation}
According to Proposition \ref{Proposition_Pestimate} and Lemmas \ref{lemma03}, \ref{lemma04}, \ref{lemma05} and \ref{lemma202147},  it holds
\begin{align}
\|e^{k}\|&\leq 4\exp(8C_5t_{k-1})\brab{\sum_{l=1}^{k}\|\sum_{j=1}^l\theta_{l-j}^{(l)}(R_t^j+R_f^j)\|}\nonumber\\
&\leq 4\exp(8C_5t_{k-1})\Big(\sum_{j=1}^kp^{(k)}_{k-j}\|R_f^j\|+\sum_{j=2}^kp^{(k)}_{k-j}\|R_t^j\|+p^{(k)}_{k-1}\|R_t^1\|\Big)\nonumber\\
&\leq 4\mathcal{M}\exp(8C_5T)\Big(C_f T+\frac32\mathcal{M}T+\mathcal{M}\Big)\tau^2 :=C_3\tau^2.\label{035}
\end{align}
Similarly, set $n=l$ in (\ref{020}) and  taking inner product with   $\nabla_{\tau} e^l$ on both sides of (\ref{020}), one has
\begin{align}
(\mathcal{D}_2 e^l,\nabla_{\tau} e^l)&=(\Delta e^l,\nabla_{\tau} e^l)+(R_t^l+R_f^l+E_1^l,\nabla_{\tau} e^l) \nonumber\\
&=-(\nabla e^l,\nabla_{\tau} \nabla e^l)+(R_t^l+R_f^l+E_1^l,\nabla_{\tau} e^l ). \label{036}
\end{align}
Summing $l$ from 1 to $k$ on \eqref{036}, using \eqref{EQ_bpositive} and the identity $2a(a-b)  = a^2 - b^2 + (a-b)^2$, we have
\begin{align} \label{ing}
\frac{\delta}{20}\sum_{l=1}^k\frac{\|\nabla_\tau e^l\|}{\tau_l}+\|\nabla e^k\|^2-\|\nabla e^0\|^2+\sum_{l=1}^k\|\nabla_\tau \nabla e^l\|^2\leq 2\sum_{l=1}^k (R_t^l+R_f^l+E_1^l,\nabla_{\tau} e^l ).
\end{align}
It follows from the Young's inequality that
\begin{align} \label{ing1}
2\sum_{l=1}^k (R_t^l+R_f^l+E_1^l,\nabla_{\tau} e^l )\leq \sum_{l=1}^k \Big(\frac{\delta}{20}\frac{\|\nabla_\tau e^l\|^2}{\tau_l}+\frac{20\tau_l}{\delta}\|R_t^l+R_f^l+E_1^l\|^2\Big).
\end{align}
Applying \eqref{ing1} and $\nabla e^0=0$ to \eqref{ing},  we arrive at
\begin{align}
\|\nabla e^k\|^2+\sum_{l=1}^k\|\nabla_\tau \nabla e^l\|^2&\leq \frac{20}{\delta}\sum_{l=1}^k\tau_l\|R_t^l+R_f^l+E_1^l\|^2 \nonumber\\
&\leq  \frac{60}{\delta}\Big(\tau_1\|R_t^1\|^2+\sum_{l=2}^k\tau_l\|R_t^l\|^2+\sum_{l=1}^k\tau_l(\|R_f^l\|^2+\|E_1^l\|^2)    \Big). \label{Ek}
\end{align}
Applying  the estimates  \eqref{027} and \eqref{035}, Lemmas \ref{lemma05} and  \ref{lemma202147} to  the inequality \eqref{Ek}, one yields
\begin{align}\label{035.1}
\|\nabla e^k\|^2+\sum_{l=1}^k\|\nabla_\tau \nabla e^l\|^2&\leq \frac{60}{\delta}\tau^3 \Big(\frac{\mathcal{M}^2}{4}+\frac94\mathcal{M}^2T\tau+TC_f^2\tau+2TC_5^2C_3^2\tau\Big)\nonumber\\
&\leq \frac{60}{\delta}\tau^3 \Big(\frac{\mathcal{M}^2}{4}+\frac94\mathcal{M}^2T+TC_f^2+2TC_5^2C_3^2\Big):=C_4^2\tau^3,
\end{align}
where $\tau\leq 1$ in \eqref{EQ_tau**}  is used.
Thus, the estimates \eqref{035.1}  and \eqref{035} yield the following $H^1$ norm estimate
\begin{align}
\|e^k\|_{H^1}\leq \sqrt{C_3^2+C_4^2}\tau^\frac32:=C_1\tau^\frac32.\label{EQ_H1e}
\end{align}

In the remainder, we will prove $\|U^k\|_{L^\infty}\leq C_\Omega (\mathcal{M}+1)$. To do so,  we need to estimate $ \|\Delta e^k\|$ due to the facts that $|e^k|_2\leq \tilde{C} \|\Delta e^k\|$ (here $|\cdot|_2$ denote the semi-norm of $H^2$-norm)  and
\begin{align}
\|U^k\|_{L^\infty}\leq  \|u^k\|_{L^\infty}+\|e^k\|_{L^\infty}\leq C_\Omega \mathcal{M}+C_\Omega \|e^k\|_{H^2},\label{0260}
\end{align}
where the embedding theorem is used in the last inequality. We now consider the  estimate of $\|\Delta e^k\|$ by taking inner product with $-\nabla_\tau\Delta e^l$ on both sides of \eqref{020} (set $n=l$), and have
\begin{align}\label{0363}
(\mathcal{D}_2\nabla e^l,\nabla_\tau\nabla e^l)+(\Delta e^l,\nabla_\tau\Delta e^l)=(R_t^l+R_f^l+E_1^l,-\nabla_\tau\Delta e^l).
\end{align}
Summing $l$ from 1 to $k$ on \eqref{0363}, using { $2(a-b)a= a^2-b^2+(a-b)^2$}, the positiveness \eqref{EQ_bpositive} and $\Delta e^0=0$, we have
\begin{equation}\label{In1}
{\sum_{l=1}^k \|\nabla_\tau \Delta e^l\|^2+}\|\Delta e^k\|^2\leq  2\sum_{l=1}^{k}\|R_t^l+R_f^l+E_1^l\|\|\nabla_\tau\Delta e^l\|,
\end{equation}
Applying the Young's inequality to \eqref{In1}, one has
\begin{align}
\|\Delta e^k\|^2&\leq  \sum_{l=1}^{k}\|R_t^l+R_f^l+E_1^l\|^2.\label{EQ_0110}
\end{align}
According to \eqref{027} and \eqref{035} and Lemmas \ref{lemma05}  and  \ref{lemma202147}, one has

\begin{align}
\sum_{l=1}^{k}\|R_t^l+R_f^l+E_1^l\|^2
&\leq  3\big(\|R_t^1\|^2+\sum_{l=2}^{k}\|R_t^l\|^2+\sum_{l=1}^{k}(\|R_f^l\|^2+\|E_1^l\|^2)\big)\nonumber\\
&\leq \big(\frac34\mathcal{M}^2+\frac{27}{4}\mathcal{M}^2t_k\tau+3C_f^2t_k\tau+3C_5^2C_3^2k\tau^2\big)\tau^2\nonumber\\
&\leq \big(\frac34\mathcal{M}^2+\frac{27}{4}\mathcal{M}T+3C_f^2T+3C_5^2C_3^2\hat{C}\big)\tau^2:= C_7^2\tau^2,
\label{EQ_deltae11}
\end{align}
where   one uses the maximum time-step assumption  \Ass{2} and \eqref{EQ_tau**} in the last inequality. Inserting \eqref{EQ_deltae11} into \eqref{EQ_0110}, we derive
\begin{align}
\|\Delta e^k\|\leq C_7\tau.\label{EQ_deltae}
\end{align}
Thus, from the estimates \eqref{EQ_deltae} and \eqref{EQ_H1e},  the condition \eqref{EQ_tau**} and the facts that $|e^k|_2\leq \tilde{C} \|\Delta e^k\|$,  we arrive at
\begin{align}\label{025}
\|e^k\|_{H^2}\leq C_8 \tau,
\end{align}
where $C_8=\sqrt{C_1^2+\tilde{C}^2C_7^2}$. Thus, combining \eqref{025}, \eqref{0260} and  \eqref{EQ_tau**}, one has
\begin{equation}\label{026}
\|U^k\|_{L^\infty}\leq  \|u^k\|_{L^\infty}+\|e^k\|_{L^\infty}
\leq C_\Omega \Big(\mathcal{M}+  C_8\tau\Big)\leq C_\Omega (\mathcal{M}+1) .
\end{equation}
Therefore,  (\ref{022}) and (\ref{024}) hold for $n=k$, i.e., the estimates \eqref{022} and \eqref{024} are proved.

The last claim \eqref{023} can be  proved  based on the result \eqref{026}.  With the help of Cauchy inequality,  from  \eqref{035} and \eqref{035.1} and the fact $|\nabla_\tau  e^l|_2\leq \tilde{C}\|\nabla_\tau \Delta e^l\|$, we have
\begin{align}\label{043}
\big(\sum_{l=1}^n\|\nabla_\tau e^l\|_{H^2}\big)^2 &\leq n \sum_{l=1}^n\|\nabla_\tau e^l\|_{H^2}^2\leq n\sum_{l=1}^n\|\nabla_\tau e^l\|^2+n\sum_{l=1}^n\|\nabla_\tau \nabla e^l\|^2+n\tilde{C}^2\sum_{l=1}^n{\|\nabla_\tau \Delta e^l\|^2}\nonumber\\
 &\leq 4C_3^2N^2\tau^4+NC_4^2\tau^3+N\tilde{C}^2\sum_{l=1}^n{\|\nabla_\tau \Delta e^l\|^2} \; .
\end{align}
We now estimate $\sum_{l=1}^n\|\nabla_\tau \Delta e^l\|^2$. {Applying the Young's inequality  $2ab\leq 8a^2+\frac12 b^2$ and inequality \eqref{EQ_deltae11} to \eqref{In1}, one has
\begin{align}\label{EQ_1630}
\sum_{l=1}^k \|\nabla_\tau \Delta e^k\|^2\leq  16\sum_{l=1}^{k}\|R_t^l+R_f^l+E_1^l\|^2\leq C_7^2\tau^2.
\end{align}}
Inserting \eqref{EQ_1630} into \eqref{043}, then it follows from the condition \Ass{2} and \eqref{EQ_tau**} that
\begin{align*}
\big(\sum_{l=1}^n\|\nabla_\tau e^l\|_{H^2}\big)^2\leq  4C_3^2\hat{C}^2+C_4^2\hat{C}+C_7^2\tilde{C}^2\hat{C}:=C_{2}^2.
\end{align*}
The proof is completed.
\end{proof}
\subsection{Analysis of the fully discrete scheme }
We now consider the boundedness of the fully discrete solution $U_h^n$, which plays the key role of the proof of Theorem \ref{theorem03}.  Define the Ritz projection operator $R_h: H^1_0(\Omega)\rightarrow V_h\subseteq H^1_0(\Omega)$ by
\begin{align}
(\nabla(v-R_hv),\nabla\omega)=0,\quad \forall \omega\in V_h.\label{EQ_or}
\end{align}
The Ritz projection $R_h$ has the following estimate (for example, see \cite[Lemma 1.1]{T06})
\begin{align}\label{045}
\|v-R_hv\|+h\|\nabla(v-R_hv)\|\leq C_\Omega h^s\|v\|_{H^s}, \quad \forall v\in H^s(\Omega)\cap H^1_0(\Omega), 1\leq s\leq r+1.
\end{align}
{Thanks to the boundedness of the semi-discrete  solutions  and   \eqref{EQ_1444}, we only need to  estimate $\|U^n-U_h^n\|_{L^\infty}$}, which can be split by the Ritz projection into the following two parts
\begin{align}\label{046}
U^n-U_h^n=U^n-R_hU^n+R_hU^n-U_h^n:=\eta^n+\xi^n,\quad n=0,1,\ldots,N.
\end{align}
For the projection error $\eta^n$, it follows from \eqref{045} that
\begin{align}
\|\eta^n\|\leq C_\Omega h^s\|U^n\|_{H^{s}},\quad 0\leq s\leq r+1.\label{EQ_02030218}
\end{align}
Thus, in remainder of this subsection,  the focus is on the estimate of the second term $\xi^n$. To do so,  we now consider the weak form of the semi-discrete equation (\ref{018}) given as
\begin{align}\label{047}
( \mathcal{D}_2 U^n,v_h)=-(\nabla U^n,\nabla v_h)+(f(U^{n-1})+f'(U^{n-1}) \nabla_\tau U^n,v_h), \quad \forall v_h\in H^1_0(\Omega).
\end{align}
Subtracting (\ref{008}) from (\ref{047}), one can use the orthogonality \eqref{EQ_or} to obtain
\begin{align}\label{048}
( \mathcal{D}_2 \xi^n,v_h)=-(\nabla \xi^n,\nabla v_h)-( \mathcal{D}_2 \eta^n,v_h)+(E_2^n,v_h),\quad\forall v_h\in V_h \, ,
\end{align}
where
\begin{align}\label{046}
E_2^n=f(U^{n-1})+f'(U^{n-1})\nabla_\tau U^n -\Big (f(U_h^{n-1})+f'(U_h^{n-1}) \nabla_\tau U_h^n \Big).
\end{align}
 Let $n=j$ in (\ref{048}).  Then,  {multiplying  $\theta_{l-j}^{(l)}$ by  (\ref{048})}, summing $j$ from 1 to $l$, and using the property \eqref{006}, one has
 \begin{align}\label{050}
 (\nabla_\tau \xi^l,v_h)=-(\sum_{j=1}^l\theta_{l-j}^{(l)}\nabla \xi^j,\nabla v_h)-(\nabla_\tau \eta^l,v_h)+(\sum_{j=1}^l\theta_{l-j}^{(l)}E_2^j,v_h) .
\end{align}
\begin{theorem}\label{theorem02}
Assume the semi-discrete scheme (\ref{018}) has a unique solution $U^n\; (n=1,\dots,N)$. Then the finite element system defined in (\ref{008}) has a unique solution $U_h^n \; (n=1,\dots,N)$, and there exist $\tau^{***}=\min\{\tau^{**},1/(8C_{11})\}$ and $h^{**}=(C_{10}C_\Omega)^{-\frac{2}{4-d}}$ such that when $\tau<\tau^{***}$ and $h<h^{**}$, it holds
\begin{align}
&\|\xi^n\| \leq C_{10} h^2,\label{051}\\
&\|U_h^n\|_{L^\infty} \leq  Q,\label{052}
\end{align}
where $Q=\bar{C}(\mathcal{M}+1)+1$, $C_{10},C_{11},C_\Omega$ are  constants independent of $\tau$.
\end{theorem}
\begin{proof}
%
It is obvious that the solution $U_h^n$ of (\ref{008}) uniquely exists   since  the coefficients matrice of (\ref{008}) is diagonally dominant. Similar to the proof of Theorem \ref{theorem01}, the mathematical induction is also used to  prove (\ref{051}) and \eqref{052}.  At the beginning, it is easy to check $\xi^0+\eta^0=U^0-U^0_h=0$  by \eqref{046},  and the inequality   (\ref{051}) also holds for $n=0$ according to  \eqref{045}. Assume (\ref{051}) holds for all $n\leq k-1$. It is known that $\|R_h v\|_{L^\infty}\leq \bar{C}\|v\|_{H^2}$ for  any $v\in H^2(\Omega)$. Thus, from Theorem \ref{theorem01}, the projection $R_h U^n$ is bounded in the sense of $L^\infty$-norm, namely
\begin{align*}
\|R_h U^n\|_{L^\infty}\leq \bar{C}(\mathcal{M}+1).
\end{align*}
Together with \eqref{EQ_02030218}, it holds for $n\leq k-1$ and $h\leq h^{**}=(C_{10}C_\Omega)^{-\frac{2}{4-d}}$ that
\begin{equation}\label{053}
\begin{split}
\|U_h^{n}\|_{L^\infty}&\leq \;\|R_hU^{n}\|_{L^\infty}+\|\xi^{n}\|_{L^\infty} \; \leq  \;\|R_hU^{n}\|_{L^\infty}+C_\Omega h^{-\frac{d}{2}}\|\xi^{n}\|\\
&\leq \; \|R_hU^{n}\|_{L^\infty}+C_\Omega C_{10}h^{-\frac{d}{2}}h^2 \; \leq \; \|R_hU^{n}\|_{L^\infty}+1 \; \leq Q.
\end{split}
\end{equation}
Due to the boundedness of $||U^{k-1}||_{L^\infty}$ and $||U_h^{k-1}||_{L^\infty}$, one has
\begin{align}
\|E_2^k\|&\leq \|f(U^{k-1})+f'(U^{k-1})\nabla_\tau U^k-(f(U_h^{k-1})+f'(U_h^{k-1})\nabla_\tau U_h^k)\|\nonumber\\
&\leq \|f(U^{k-1})-f(U_h^{k-1})\|+\|(f'(U^{k-1})-f'(U_h^{k-1}))U^k\| +\|f'(U_h^{k-1})(U^k-U_h^k)\| \nonumber\\
&\quad+\|(f'(U^{k-1})-f'(U_h^{k-1}))U^{k-1}\| +\|f'(U_h^{k-1})(U^{k-1}-U_h^{k-1})\|\nonumber\\
&\leq C_{11}(\|U^{k-1}-U_h^{k-1}\|+\|U^k-U_h^k\|)\nonumber\\
&\leq C_{11}(\|\xi^{k-1}\|+\|\xi^k\|+\|\eta^{k-1}\|+\|\eta^k\|)\nonumber\\
&\leq C_{11}\Big(\|\xi^{k-1}\|+\|\xi^k\|+2C_\Omega (\mathcal{M}+1)h^2\Big),\label{054}
\end{align}
where one uses \eqref{EQ_02030218} in the last inequality above, and
$$C_{11}:=3\sup_{\lvert v \rvert\leq \max{(C_\Omega (\mathcal{M}+1),Q)}}\lvert f' (v)\rvert + 2C_\Omega (\mathcal{M}+1)\sup_{\lvert v \rvert\leq \max{(C_\Omega (\mathcal{M}+1),Q)}}\lvert f'' (v)\rvert .$$
Taking  $v_h=\xi^l$ in \eqref{050}  and summing $l$ from $1$ to $k$, one has
\begin{align}
\sum_{l=1}^k(\nabla_\tau \xi^l,\xi^l)&=-\sum_{l=1}^k(\sum_{j=1}^l\theta_{l-j}^{(l)}\nabla \xi^j,\nabla\xi^l)-\sum_{l=1}^k(\nabla_\tau \eta^j,\xi^l)+\sum_{l=1}^k(\sum_{j=1}^l\theta_{l-j}^{(l)}E_2^j,\xi^l) \nonumber\\
&\leq \sum_{l=1}^k\|\xi^l\|\|\nabla_\tau \eta^l\|+\sum_{l=1}^k\|\xi^l\|\sum_{j=1}^l\theta_{l-j}^{(l)}\|E_2^j\|,\label{055}
\end{align}
where the last inequality uses the positive definiteness of \eqref{EQ_bpositive}.
Inserting \eqref{054} into \eqref{055}, one has
\begin{align*}
\|\xi^k\|^2 \leq\|\xi^0\|^2+ 2\sum_{l=1}^k\|\xi^l\|\|\nabla_\tau \eta^l\|
+2C_{11}\sum_{l=1}^k\|\xi^l\|\sum_{j=1}^l\theta_{l-j}^{(l)}\Big(\|\xi^{j-1}\|+\|\xi^{j}\|+2C_\Omega (\mathcal{M}+1) h^2\Big).
\end{align*}
Choosing  $k^*(0\leq k^*\leq k)$ such that $\|\xi^{k^*}\|=\max_{0\leq i\leq k}\|\xi^{i}\|$. Then the above inequality {combining with \eqref{011}} yield
\begin{align*}
\|\xi^{k^*}\|^2&\leq \|\xi^0\|\|\xi^{k^*}\|+2\|\xi^{k^*}\|\sum_{l=1}^{k^*}\|\nabla_\tau \eta^l\|
+4C_{11}\|\xi^{k^*}\|\Big(C_\Omega (\mathcal{M}+1)Th^2+\sum_{l=1}^{k^*}\tau_l\|\xi^l\| \Big).
\end{align*}
Thus, one further has
\begin{align*}
\|\xi^{k}\|\leq \|\xi^0\|+ 2\sum_{l=1}^k\|\nabla_\tau \eta^l\|
+4C_{11}\Big(C_\Omega (\mathcal{M}+1)Th^2+\sum_{l=1}^k\tau_l\|\xi^l\|\Big).
\end{align*}
According to { \eqref{EQ_02030218}} and \eqref{023}, we find
\begin{align}
\|\xi^{k}\|&\leq C_\Omega h^2\| u^0\|_{H^2}+ 2 C_\Omega h^2\sum_{l=1}^k\|\nabla_\tau U^l\|_{H^2}+4C_{11}\Big(C_\Omega (\mathcal{M}+1)Th^2+\sum_{l=1}^k\tau_l\|\xi^l\|\Big)\nonumber\\
&{\leq (\mathcal{M}+2 \sum_{l=1}^k(\|\nabla_\tau u^l\|_{H^2}+\|\nabla_\tau e^l\|_{H^2}))C_\Omega h^2+4C_{11} \Big((\mathcal{M}+1)TC_\Omega h^2+\sum_{l=1}^k\tau_l\|\xi^l\|\Big)}\label{EQ_1857}\\
&\leq \Big(\mathcal{M}+2 ( \mathcal{M}T+C_{2}+4C_{11} (\mathcal{M}+1)T)\Big)C_\Omega h^2+4C_{11}\sum_{l=1}^k\tau_l\|\xi^l\|\nonumber \\
&:= C_{12} h^2+4C_{11}\sum_{l=1}^k\tau_l\|\xi^l\|,\nonumber
\end{align}
where $\|\xi^0\|=\|R_h U^0-U_h^0\|=\|R_h u^0-u^0\|\leq C_\Omega h^2\| u^0\|_{H^2}$ is used.
 Noting  $\tau\leq\tau^{***}\leq 1/(8C_{11})$, we have
$$
\|\xi^{k}\|\leq 2C_{12} h^2+8C_{11}\sum_{l=1}^{k-1}\tau_l\|\xi^l\|.
$$
Thus, from Lemma \ref{lemma04}, we arrive at
$$
\|\xi^{k}\|
\leq 2\exp(8C_{11}T)C_{12} h^2:=C_{10} h^2.
$$
Hence, for $h\leq h^{**}$, it holds
\begin{align}\label{061}
\|U_h^k\|_{L^\infty}&\leq \|R_hU^k\|_{L^\infty}+\|\xi^k\|_{L^\infty}\leq \|R_hU^k\|_{L^\infty}+C_\Omega h^{-\frac{d}{2}}\|\xi^k\|\leq Q.
\end{align}
Therefore, the estimates (\ref{051}) and (\ref{052}) hold for $n=k$ and the proof is completed.
\end{proof}
\section{Unconditionally optimal $L^2$-norm  error estimate}\label{Section_4}
We now present the unconditionally optimal $L^2$-norm error estimate for fully discrete scheme \eqref{008}.
\begin{theorem}\label{theorem03}
Assume $ u(\cdot,t)\in H^{r+1}(\Omega)\bigcap H_0^1(\Omega)$ for any $t\geq 0$ and $r\geq 1$, and the conditions \Ass{1} and \Ass{2} hold.
Then there exist constants $\tau^*$ and $h^*$ such that, when $\tau<\tau^*$ and $h<h^*$,  the $r$-degree finite element system defined in \eqref{008} owns a unique solution and
satisfies
\begin{align}\label{010}
\|u^n-U_h^n\| \leq C_0^* (\tau^2+h^{r+1}),
\end{align}
where $C_0^*$ is a positive constant independent of $h$ and $\tau$.
\end{theorem}
\begin{proof}
The error of fully discrete solution and exact solution can be split into the following two parts
\begin{align}
\|u^n-U_h^n\|\leq \|u^n-R_hu^n\|+\|R_hu^n-U_h^n\|:=\|\vartheta^n\|+\|\zeta^n\|.\label{EQ_2103252033}
\end{align}
Note that the projection error $\vartheta^n$  can be immediately  estimated by \eqref{045} as follows
\begin{align}
\|\vartheta^n\|\leq C_\Omega h^s\|u^n\|_{H^{s}},\quad 0\leq s\leq r+1.\label{EQ_02031547}
\end{align}
Thus, we only need to consider $\zeta^n$.
Subtracting \eqref{008} from \eqref{001}, we get the error equation of $\zeta^n$
\begin{align}\label{062}
(\mathcal{D}_2 \zeta^n,v_h)=-(\nabla \zeta^n,\nabla v_h)-(\mathcal{D}_2 \vartheta^n,v_h)+(E_3^n,v_h)+(R_t^n+R_f^n,v_h),\quad \forall v_h\in V_h,
\end{align}
where $E_3^n=f(u^{n-1})+f'(u^{n-1}) \nabla_\tau u^n-(f(U_h^{n-1})+f'(U_h^{n-1}) \nabla_\tau U_h^n)$.\par
Let $n=j$ in \eqref{062}. Multiplying  (\ref{062}) by DOC kernels $\theta_{l-j}^{(l)}$ and summing $j$ from 1 to $l$, one has
 \begin{align}\label{062.2}
(\nabla_\tau \zeta^l,v_h)=-(\sum_{j=1}^l\theta_{l-j}^{(l)}\nabla \zeta^j,\nabla v_h)-(\nabla_\tau \vartheta^l,v_h)+(\sum_{j=1}^l\theta_{l-j}^{(l)}E_3^j,v_h) +(\sum_{j=1}^l\theta_{l-j}^{(l)}(R_t^j+R_f^j),v_h),
\end{align}
where the property \eqref{006}  is used.
Taking $v_h=\zeta^l$ in (\ref{062.2}), summing $l$ from $1$ to $n$ and using the inequality \eqref{045} and Corollary \ref{corollary01}, one has
\begin{equation}\label{062.3}
\begin{split}
\sum_{l=1}^n(\nabla_\tau \zeta^l,\zeta^l) &=-\sum_{l=1}^n(\sum_{j=1}^l\theta_{l-j}^{(l)}\nabla \zeta^j,\nabla \zeta^l)-\sum_{l=1}^n(\nabla_\tau \vartheta^l,\zeta^l)+\sum_{l=1}^n(\sum_{j=1}^l\theta_{l-j}^{(l)}(E_3^j+R_t^j+R_f^j),\zeta^l)\\
&\leq \sum_{l=1}^n\|\nabla_\tau \vartheta^l\|\|\zeta^l\|+\sum_{l=1}^n\sum_{j=1}^l\theta_{l-j}^{(l)}(\|E_3^j\|+\|R_f^j\|+\|R_t^j\|)\|\zeta^l\|.
\end{split}
\end{equation}
Noting $$\|\nabla_\tau \vartheta^l\|\leq  C_\Omega \|\nabla_\tau u^l\|_{H^{r+1}}h^{r+1}\leq C_\Omega h^{r+1} \|\int_{t_{l-1}}^{t_l} \partial_t u(s)\zd s\|_{H^{r+1}}\leq C_\Omega \mathcal{M}\tau_l h^{r+1},$$
together with the inequality $2(a-b)a\geq a^2-b^2$, we arrive at
\begin{align}
\|\zeta^n\|^2&\leq \|\zeta^0\|^2+2\sum_{l=1}^n\sum_{j=1}^l\theta_{l-j}^{(l)}(\|E_3^j\|+\|R_f^j\|+\|R_t^j\|)\|\zeta^l\|+2C_\Omega \mathcal{M}h^{r+1}\sum_{l=1}^n\tau_l  \|\zeta^l\|.\label{EQ_02031706}
\end{align}
Thanks to the boundedness of $\|U^n_h\|_{L^\infty}$ in \eqref{052},  the nonlinear term $E_3^n$ can be estimated by
\begin{align}
\|E_3^n\|&=\|f(u^{n-1})+f'(u^{n-1}) \nabla_\tau u^n-(f(U_h^{n-1})+f'(U_h^{n-1}) \nabla_\tau U_h^n)\|\nonumber\\
&\leq\|f(u^{n-1})-f(U_h^{n-1})\|+\|(f'(u^{n-1})-f'(U_h^{n-1}))u^n\|+\|f'(U_h^{n-1})(u^n-U_h^n)\|\nonumber\\
&\quad+\|(f'(u^{n-1})-f'(U_h^{n-1}))u^{n-1}\| +\|f'(U_h^{n-1})(u^{n-1}-U_h^{n-1})\|\nonumber\\
&\leq C_{13}\Big(\|u^{n-1}-U_h^{n-1}\|+\|u^{n}-U_h^{n}\|\Big)\nonumber\\
&\leq C_{13}\Big(\|\zeta^{n-1}\|+\|\zeta^{n}\|+\|\vartheta^{n-1}\|+\|\vartheta^{n}\|\Big)\nonumber\\
&\leq C_{13}\Big(\|\zeta^{n-1}\|+\|\zeta^{n}\|+2C_\Omega \mathcal{M}  h^{r+1}\Big),\label{063}
\end{align}
where the last inequality holds by \eqref{045} and
$$C_{13}=2\sup_{|v|<\max\{C_\Omega \mathcal{M},Q\}}|f'(v)|+2\sup_{|v|<\max\{C_\Omega \mathcal{M},Q\}}|f''(v)|C_\Omega \mathcal{M}.$$
  Inserting \eqref{063} into \eqref{EQ_02031706},  one has
  \begin{align*}
\|\zeta^n\|^2&\leq \|\zeta^0\|^2+2 C_{13}\sum_{l=1}^n\sum_{j=1}^l\theta_{l-j}^{(l)}(\|\zeta^{j-1}\|+\|\zeta^{j}\|)\|\zeta^l\|\\
&+2\sum_{l=1}^n\sum_{j=1}^l\theta_{l-j}^{(l)}(\|R_f^j\|+\|R_t^j\|)\|\zeta^l\|
+2C_\Omega \mathcal{M}(1+2C_{13})h^{r+1}\sum_{l=1}^n\tau_l  \|\zeta^l\|.
\end{align*}
Choosing  $n^*$ such that  $\|\zeta^{n^*}\|=\max_{0\leq l\leq n}\|\zeta^l\|$. The above inequality yields
\begin{align*}
\|\zeta^{n^*}\|^2&\leq \|\zeta^0\|\|\zeta^{n^*}\|+4 C_{13}\sum_{l=1}^{n^*}\sum_{j=1}^l\theta_{l-j}^{(l)}\|\zeta^{n^*}\|\|\zeta^{l}\|\\
&+2\sum_{l=1}^{n^*}\sum_{j=1}^l\theta_{l-j}^{(l)}(\|R_f^j\|+\|R_t^j\|)\|\zeta^{n^*}\|
+2C_\Omega \mathcal{M}(1+2C_{13})h^{r+1}\sum_{l=1}^{n^*}\tau_l  \|\zeta^{n^*}\|.
\end{align*}
Eliminating a $\|\zeta^{n^*}\|$ from both sides and using the facts that $\|\zeta^{n}\|\leq \|\zeta^{n^*}\|$ and $ n^*\leq n$, we arrive at
\begin{align}
\|\zeta^{n}\|\leq \|\zeta^{n^*}\|
&\leq \|\zeta^0\|+4 C_{13}\sum_{l=1}^{n}\tau_l\|\zeta^{l}\|+2\sum_{l=1}^{n}\sum_{j=1}^l\theta_{l-j}^{(l)}(\|R_f^j\|+\|R_t^j\|)
+2C_\Omega \mathcal{M}(1+2C_{13})t_n h^{r+1}\nonumber\\
&\leq 4 C_{13}\sum_{l=1}^{n}\tau_l\|\zeta^{l}\|+2\sum_{l=1}^{n}\sum_{j=1}^l\theta_{l-j}^{(l)}(\|R_f^j\|+\|R_t^j\|)
+C_\Omega \mathcal{M}(2T+4C_{13}T+1) h^{r+1},\label{EQ_1545}
\end{align}
where   $\|\zeta^0\|=\|R_h u_0-U_h^0\|=\|R_h u_0-u_0\|\leq C_\Omega h^{r+1}\| u_0\|_{H^{r+1}}$ is used.
By exchanging the order of summation and using  identity \eqref{EQ_ptheta}, Lemmas \ref{lemma05}, \ref{lemma202147} and Proposition \ref{Proposition_Pestimate}, it holds
\begin{align}
2\sum_{l=1}^{n}&\sum_{j=1}^l\theta_{l-j}^{(l)}(\|R_f^j\|+\|R_t^j\|)=2\sum_{j=1}^np^{(n)}_{n-j}(\|R_f^j\|+\|R_t^j\|) \nonumber\\
&=2\sum_{j=1}^np^{(n)}_{n-j}\|R_f^j\|+2\sum_{j=2}^np^{(n)}_{n-j}\|R_t^j\|+p^{(n)}_{n-1}\|R_t^1\| \nonumber\\
&\leq (2C_ft_n+3t_n\mathcal{M}+\mathcal{M})\tau^2. \label{EQ_1615}
\end{align}
Thus, inserting \eqref{EQ_1615} into \eqref{EQ_1545}, for $\tau \leq \tau^*\leq 1/(8C_{13})$, we have
\begin{align*}
\|\zeta^{n}\| \leq 8 C_{13}\sum_{l=1}^{n-1}\tau_l\|\zeta^{l}\|+2(2C_fT+3T\mathcal{M}+\mathcal{M})\tau^2+2C_\Omega \mathcal{M}(2T+4C_{13}T+1) h^{r+1},
\end{align*}
which together with Lemma \ref{lemma04} implies that
\begin{align}
\|\zeta^{n}\|&\leq 2\exp(8 C_{13}T) \big((2C_fT+3T\mathcal{M}+\mathcal{M})\tau^2+C_\Omega \mathcal{M}(2T+4C_{13}T+1) h^{r+1}\big)\nonumber\\
&\leq C_{14}(h^{r+1}+\tau^2),\label{EQ_03252036}
\end{align}
where
$C_{14}=2\exp(8 C_{13}T)\max\{2C_fT+3T\mathcal{M}+\mathcal{M},C_\Omega \mathcal{M}(2T+4C_{13}T+1)\}.$
The proof is completed by inserting \eqref{EQ_02031547} and \eqref{EQ_03252036} into \eqref{EQ_2103252033} and setting $C_0^*:=C_\Omega \mathcal{M}+C_{14}$.
\end{proof}

\section{Numerical Examples}\label{Section_5}
We now present two examples to investigate the quantitative accuracy of fully discrete scheme \eqref{008} from two perspectives: (a) the convergence order of numerical scheme \eqref{008} in time and space; (b) the unconditional convergence by fixing the spacial size $h$ and refining the temporal size $\tau$.  To obtain the variable time steps, we construct the time steps by $\tau_k = T\lambda_k/\Lambda$ for $1\leq k\leq N$, where $\Lambda = \sum_{k=1}^N \lambda_k$ and $\lambda_k$ is randomly drawn from the uniform distribution on $(0,1)$. In the simulations, we only use the linear finite element (i.e. $r=1$), and consider the time steps in two cases:
\begin{itemize}
\item[]Case (i): the ratios of adjacent time steps satisfy \Ass{1}, i.e., $0<r_k<4.8645$;
\item[]Case (ii): the ratios do not satisfy \Ass{1}, i.e., can be taken large randomly.
\end{itemize}


\noindent {\bf Example 5.1.}  We here consider the computation of the following 2D nonlinear parabolic equation
$$\partial_t u=\Delta u + \sqrt{1+u^2} + g(\bm{x},t) .$$
In the simulations, we take the final time $T=1$ and the computational domain $\Omega = (0,1)^2$. As a benchmark solution, we take an exact solution in the form of $u(\bm{x},t)=(1+t^3)x_1(1-x_1)^2x_2(1-x_2)^2$, and the source term $g(\bm{x},t)$ can be calculated accordingly.

Table \ref{table3} shows the spatial $L^2$-error by increasing $M$ and fixing $N=10^4$. Tables \ref{table1} and \ref{table2} show the temporal $L^2$-errors by taking $M=N$ and increasing $N$ under Cases (i) and (ii) in above, respectively.  From Tables \ref{table3}-\ref{table2}, the second order convergence rates can be observed, which agrees with the results in Theorem \ref{theorem03}.
In addition,  The left panel in Figure \ref{fig.1} plots the $L^2$-error by fixing $N$ and increasing $M$, which implies that the error estimate is unconditionally stable since the solution does not blowup for any temporal and spatial ratios.



\begin{table}[ht]
\begin{center}
\caption {Errors and spatial convergence orders with $N=10^4$ for Example 5.1} \vspace*{0.5pt}\label{table3}
\def\temptablewidth{1.0\textwidth}
{\rule{\temptablewidth}{1pt}}
\begin{tabular*}{\temptablewidth}{@{\extracolsep{\fill}}cccc}
 $M$ &\multicolumn{1}{c}{error}&\multicolumn{1}{c}{order} &\multicolumn{1}{c}{$r_{\max}$} \\
\hline
40    &4.2893e-05   &--       &4.6836       \\
80    &8.9895e-06     &2.2544      &4.8091      \\
160    &1.9913e-06   &2.1745       &4.7541   \\
320    &4.6221e-07    &2.1071       &4.7043       \\
\end{tabular*}
{\rule{\temptablewidth}{1pt}}
\end{center}
\end{table}

\begin{table}[ht]
\begin{center}
\caption {Errors and time convergence orders with $0<r_k<4.8645$ for Example 5.1 } \vspace*{0.5pt}\label{table1}
\def\temptablewidth{1.0\textwidth}
{\rule{\temptablewidth}{1pt}}
\begin{tabular*}{\temptablewidth}{@{\extracolsep{\fill}}cccc}
 $N$ &\multicolumn{1}{c}{error}&\multicolumn{1}{c}{order} &\multicolumn{1}{c}{$r_{\max}$} \\
\hline
120    &3.6800e-06      &--      &4.2785        \\
240    &8.3695e-07    &2.1365      &4.6672       \\
480    &1.9795e-07    &2.0800      &4.1304     \\
960    &4.8114e-08    &2.0406      &4.2979       \\
\end{tabular*}
{\rule{\temptablewidth}{1pt}}
\end{center}
\end{table}

\begin{figure}[htbp]
\begin{center}
\includegraphics[width=2.5in]{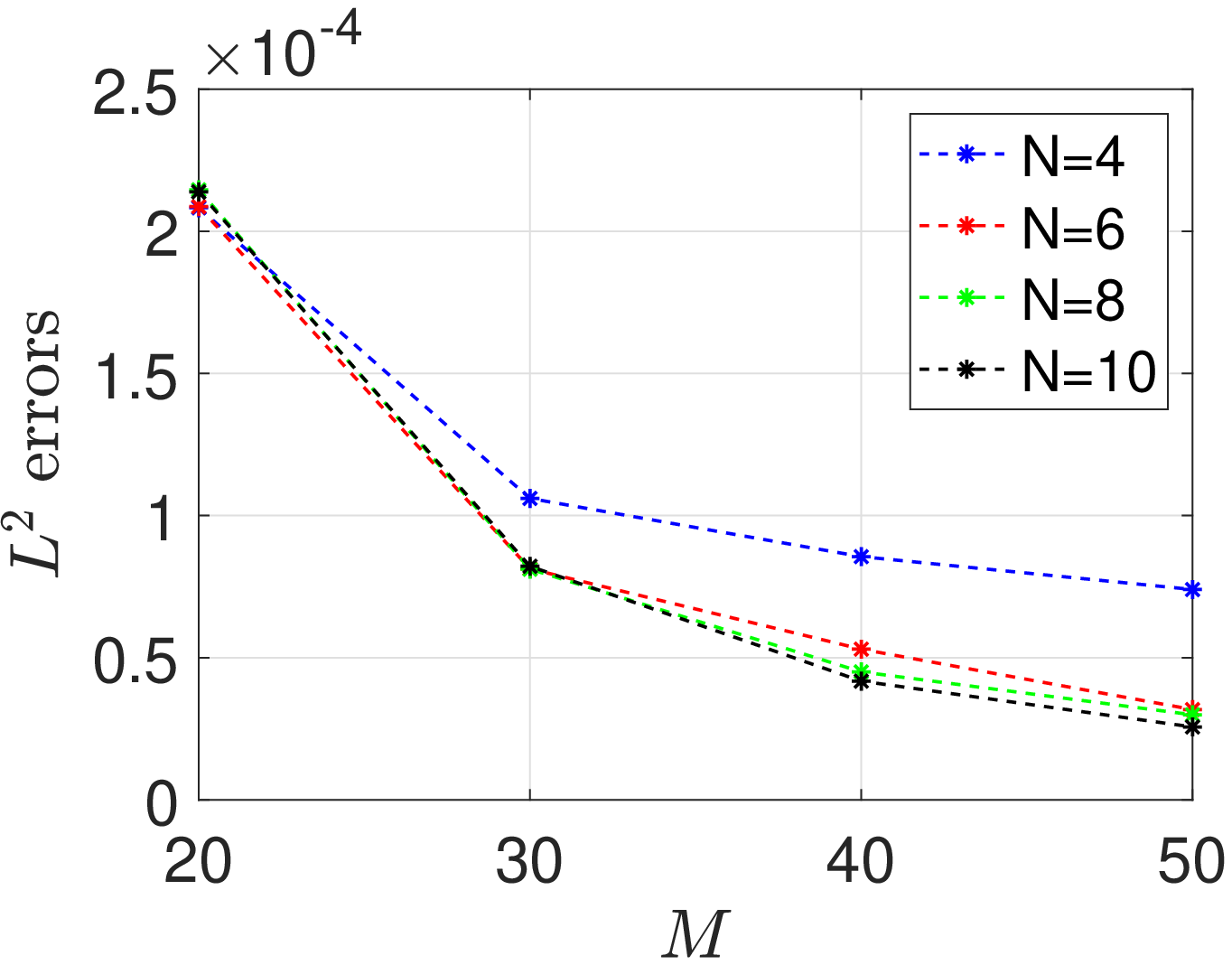}
\includegraphics[width=2.5in]{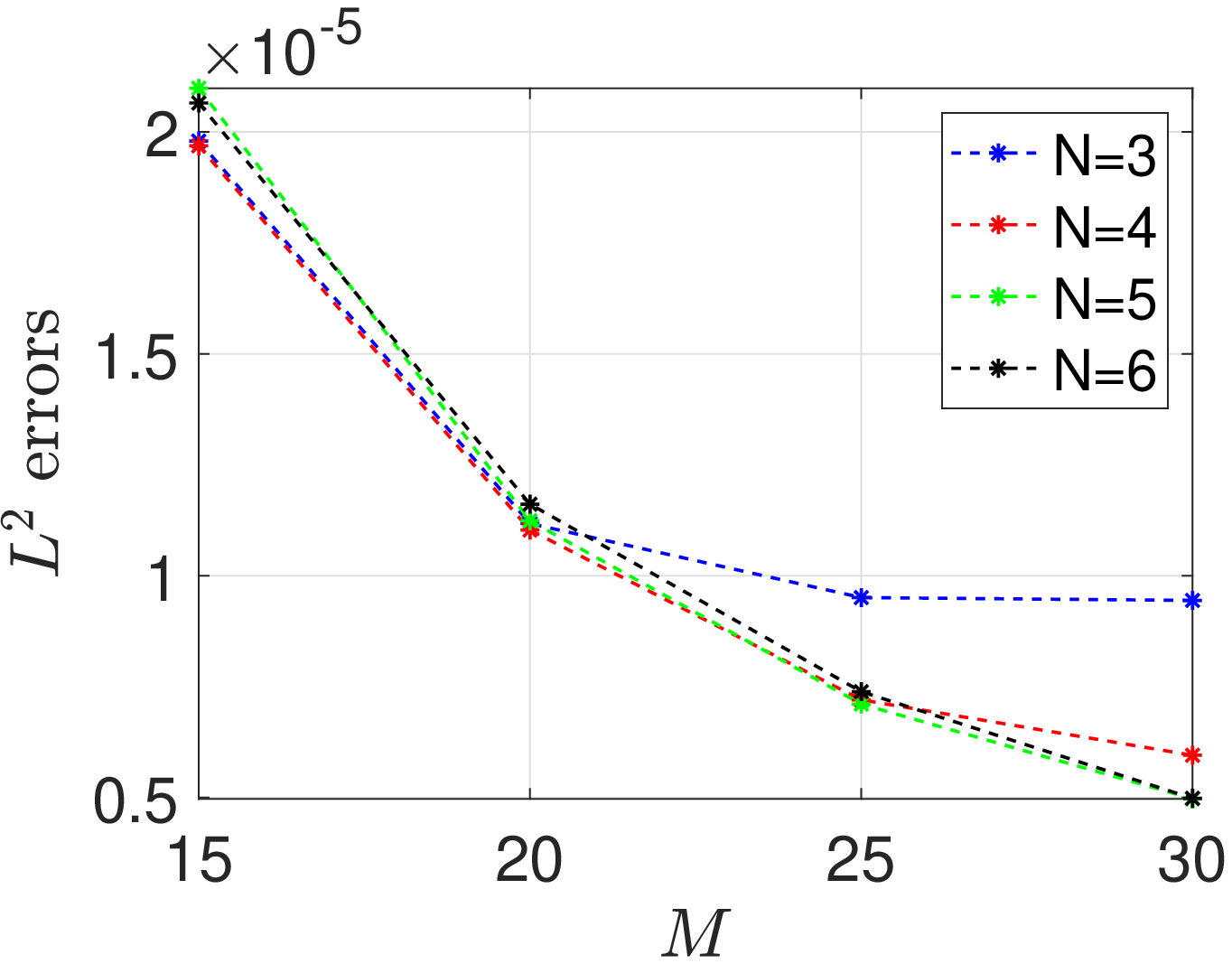}
\caption{To investigate the scheme is unconditionally stable, $L^2$-errors are plotted by increasing $M$ for different given $N$: left and right panels for Examples 5.1 and 5.2, respectively.  }\label{fig.1}
\end{center}
\end{figure}
%

\begin{table}[ht]
\begin{center}
\caption {Errors and time convergence orders with $r_k>0$ for Example 5.1 } \vspace*{0.5pt}\label{table2}
\def\temptablewidth{1.0\textwidth}
{\rule{\temptablewidth}{1pt}}
\begin{tabular*}{\temptablewidth}{@{\extracolsep{\fill}}cccc}
 $N$ &\multicolumn{1}{c}{error}&\multicolumn{1}{c}{order} &\multicolumn{1}{c}{$r_{\max}$} \\
\hline
120    &3.6533e-06    &--       &21.6746        \\
240    &8.3479e-07    &2.1297       &24.1472       \\
480    &1.9783e-07    &2.0772       &285.448    \\
960    &4.7942e-08    &2.0449       &792.551       \\
\end{tabular*}
{\rule{\temptablewidth}{1pt}}
\end{center}
\end{table}

\bigskip
\noindent {\bf Example 5.2.}  We now consider a 3D nonlinear parabolic equation
$$\partial_t u=\Delta u + u-u^3 + g(\bm{x},t).$$
In the simulations, we take the final time $T=1$ and the computational domain $\Omega = (0,1)^3$. Again, as a benchmark solution, we take an exact solution in the form of $u(\bm{x},t)=(1+t^3)x_1(1-x_1)^2x_2(1-x_2)^2x_3(1-x_3)^2$, and the source term $g(\bm{x},t)$ can be calculated accordingly.

Table \ref{table5} shows the spatial $L^2$-error by increasing $M$ and fixing $N=10^3$. Tables \ref{table4} and \ref{table6} show the temporal $L^2$-errors by taking $M=N$ and increasing $N$ under Cases (i) and (ii) in above, respectively.  From Tables \ref{table5}-\ref{table6}, the second order convergence rates can be observed, which again agrees with the results in Theorem \ref{theorem03}.
The right panel in Figure \ref{fig.1} plots the $L^2$-error by fixing $N$ and increasing $M$, which again implies that the error estimate is unconditionally stable since the solution does not blowup for any temporal and spatial ratios.

%

\begin{table}[ht]
\begin{center}
\caption {Errors and  spatial convergence orders with  $N=10^3$ for Example 5.2 } \vspace*{0.5pt}\label{table5}
\def\temptablewidth{1.0\textwidth}
{\rule{\temptablewidth}{1pt}}
\begin{tabular*}{\temptablewidth}{@{\extracolsep{\fill}}cccc}
 $M$ &\multicolumn{1}{c}{error}&\multicolumn{1}{c}{order} &\multicolumn{1}{c}{$r_{\max}$} \\
\hline
6      &1.3595e-04    &--             &4.5317         \\
12    &3.3786e-05    &2.0086     &4.3782      \\
24    &8.4315e-06    &2.0025     &4.3911      \\
48    &2.1069e-06    &2.0007     &4.6375     \\
\end{tabular*}
{\rule{\temptablewidth}{1pt}}
\end{center}
\end{table}

\begin{table}[ht]
\begin{center}
\caption {Errors and  time convergence orders with  $0<r_k<4.8645$ for Example 5.2 } \vspace*{0.5pt}\label{table4}
\def\temptablewidth{1.0\textwidth}
{\rule{\temptablewidth}{1pt}}
\begin{tabular*}{\temptablewidth}{@{\extracolsep{\fill}}cccc}
 $N$ &\multicolumn{1}{c}{error}&\multicolumn{1}{c}{order} &\multicolumn{1}{c}{$r_{\max}$} \\
\hline
4     &3.0607e-04     &--              &2.3893         \\
8    &7.5914e-05       &2.0114      &2.2410       \\
16    &  1.8918e-05    &2.0046      & 2.9583       \\
32    &4.7170e-06    &2.0038      &3.9028     \\
\end{tabular*}
{\rule{\temptablewidth}{1pt}}
\end{center}
\end{table}

\begin{table}[ht]
\begin{center}
\caption {Errors and  time convergence orders with  $r_k>0$ for Example 5.2 } \vspace*{0.5pt}\label{table6}
\def\temptablewidth{1.0\textwidth}
{\rule{\temptablewidth}{1pt}}
\begin{tabular*}{\temptablewidth}{@{\extracolsep{\fill}}cccc}
 $N$ &\multicolumn{1}{c}{error}&\multicolumn{1}{c}{order} &\multicolumn{1}{c}{$r_{\max}$} \\
\hline
4     &3.0721e-04     &--              &2.4017         \\
8    &7.5918e-05       &2.0167      &13.2425       \\
16    &  1.8929e-05    &2.0039      &32.1533       \\
32    &4.7200e-06    &2.0037      &25.1107     \\
\end{tabular*}
{\rule{\temptablewidth}{1pt}}
\end{center}
\end{table}


\section{Conclusions}\label{Section_6}
We have presented the unconditionally optimal error estimate of a linearized variable-time-step BDF2  scheme for nonlinear parabolic equations in conjunction with a Galerkin finite element approximation in space. The rigorous error estimate of $\mathcal{O}(\tau^2+h^{r+1})$ in the $L^2$-norm has been established under mild assumptions on the ratio of adjacent time steps \Ass{1} and the maximum time-step size \Ass{2}.
The analysis is based on the recently developed DOC and DCC kernels and the time-space error splitting approach.  The techniques of DOC and DCC kernels facilitate the proof of the second order convergence of BDF2 with a new ratio
 of adjacent time steps, i.e., $0<r_k <  4.8645$.
The error splitting approach divides the error estimate of numerical solution into the estimates of $\|U^n\|_{H^2}$ and $\|R_hU^n-U_h^n\|_{L^2}$, which circumvents the ratio restriction of time-space sizes, i.e., this is the so-called unconditionally optimal error estimate.

In addition, our error estimate is robust. The robustness here means the error estimate does not require any extra restriction on time steps except the conditions \Ass{1} and \Ass{2}. Meanwhile, we used the first-order  BDF1 to calculate the first level solution $u^1$, and found that this first-order scheme did not bring the loss of global accuracy of second order. Although a great progress has been made for the second order convergence of variable time-step BDF2 scheme solving nonlinear problems, but most of them use the implicit approximation to the nonlinear terms.
As far as we know, it is a pioneer work to present the optimal error estimate for the variable time-step BDF2 scheme with a linearized approximation to nonlinear terms only under the ratio condition $0<r_k< 4.8645$ and a mild assumption on maximum time step \Ass{2}. Numerical examples were provided to verify our theoretical analysis.  
\section*{Acknowledgements}
The research is supported by the National Natural Science Foundation of China (Nos. 11771035 and 12171376), 2020- JCJQ-ZD-029, NSAF U1930402. The numerical calculations in this paper have been done on the supercomputing system in the Supercomputing Center of Wuhan University.
JZ would like to thank professor Tao Tang and professor Zhimin Zhang for their valuable discussions on this topic.

\section{Appendix}\label{appendix}
\subsection{The proof of Lemma \ref{lemma05}} \label{App61}
\begin{proof}
It follows from the Taylor expansion that
\begin{align*}\label{014}
\|R_f^j\|&=\|(u^j-u^{j-1})^2\int_0^1f''(u^{j-1}+s\nabla_\tau u^j)(1-s)\zd s\|\\
&\leq\|u^j-u^{j-1}\|\cdot\|(u^j-u^{j-1})\int_0^1f''(u^{j-1}+s\nabla_\tau u^j)(1-s)\zd s\|_{L^\infty}\\
&\leq\|\int_{t_{j-1}}^{t_j}\partial_tu\zd t\|\cdot\|u^j-u^{j-1}\|_{L^\infty}\|\int_0^1f''(u^{j-1}+s\nabla_\tau u^j)(1-s)\zd s\|_{L^\infty}\\
&\leq\mathcal{M}\tau_j\cdot\big\|\int_{t_{j-1}}^{t_j}\partial_tu\zd t\big\|_{L^\infty}\frac{\sup_{|u|\leq C_\Omega \mathcal{M}}|f''(u)|}2\\
&\leq\frac12\mathcal{M}\tau_j\sup_{|u|\leq C_\Omega \mathcal{M}}|f''(u)|\int_{t_{j-1}}^{t_j}\|\partial_tu\|_{L^\infty}\zd t \; \leq\frac12\sup_{|u|\leq C_\Omega \mathcal{M}}|f''(u)|{C_\Omega }\mathcal{M}^2\tau_j^2,
\end{align*}
where $\|\partial_tu\|_{L^\infty}\leq C_\Omega \|\partial_tu\|_{H^2}\leq C_\Omega \mathcal{M}$ is used.
The proof is completed.
\end{proof}
\subsection{The proof Lemma \ref{lemma202147}} \label{App62}
\begin{proof}
Recall that the first step value $u^1$ is computed by BDF1, i.e., $R_t^1=\frac{u_1-u_0}{\tau_1}-\partial_t u(t_1)$, one has
\begin{align*}
\| R_t^1\| &= \|\frac{ u_1- u_0}{\tau_1}- \partial_t u(t_1)\| = \frac{1}{\tau_1}\|\int_0^{t_1} \partial_t u(s)-\partial_t u(t_1) \zd s\|\nonumber\\
&\leq \frac{1}{\tau_1}\int_0^{t_1}\| \partial_t u(s)- \partial_t u(t_1)\| \zd s \leq \frac{1}{\tau_1}\int_0^{t_1}\int_s^{t_1}\| \partial_{tt} u(t)\|\zd t \zd s \leq \frac{\mathcal{M}}{2}\tau_1.
\end{align*}
For $j\geq 2$,  by using the Taylor's expansion formula , one produces (also see \cite[Theorem 10.5]{T06})
\begin{align*}
R_t^j = -\frac{1+r_j}{2\tau_j}\int_{t_{j-1}}^{t_j}(t-t_{j-1})^2\partial_{ttt} u\d t+\frac{r_j}{2(1+r_j)\tau_{j-1}}\int_{t_{j-2}}^{t_j}(t-t_{j-2})^2\partial_{ttt} u\d t,\quad 2\leq j\leq N.
\end{align*}
Combining with the regularity assumption \eqref{EQ_uregularity } and the ratio condition \Ass{1}, one derive
\begin{align}
\|R_t^j\| \leq \frac{\mathcal{M}}{6}(1+r_j)\tau_j(2\tau_j+\tau_{j-1})\leq \frac{3}{2}\mathcal{M}\tau_j\tau.
\end{align}
The proof is completed.
\end{proof}

\bibliographystyle{plain}
  \bibliography{Nonlinear_ref}
\end{document}